\documentclass[12pt]{article}
\usepackage{amssymb}
\usepackage{amsthm}
\usepackage{amsmath}
\usepackage{graphicx}
\usepackage{enumerate}

\DeclareMathOperator{\Con}{Con}

\newtheorem{theorem}{Theorem}[section]
\newtheorem{definition}[theorem]{Definition}
\newtheorem{lemma}[theorem]{Lemma}

\newtheorem{example}[theorem]{Example}

\title{Algebras describing pseudocomplemented, relatively pseudocomplemented and sectionally pseudocomplemented posets}
\author{Ivan~Chajda and Helmut~L\"anger}
\date{}
\begin{document}

\footnotetext{Support of the research by the Austrian Science Fund (FWF), project I~4579-N, and the Czech Science Foundation (GA\v CR), project 20-09869L, entitled ``The many facets of orthomodularity'', as well as by \"OAD, project CZ~02/2019, entitled ``Function algebras and ordered structures related to logic and data fusion'', and, concerning the first author, by IGA, project P\v rF~2021~030, is gratefully acknowledged.}

\maketitle

\begin{abstract}
In order to be able to use methods of Universal Algebra for investigating posets, we assign to every pseudocomplemented poset, to every relatively pseudocomplemented poset and to every sectionally pseudocomplemented poset a certain algebra (based on a commutative directoid or on a $\lambda$-lattice) which satisfies certain identities and implications. We show that the assigned algebras fully characterize the given corresponding posets. It turns out that the assigned algebras satisfy strong congruence properties which can be transferred back to the posets. We also mention applications of such posets in certain non-classical logics.
\end{abstract}



\section{Introduction}

For investigating algebras researchers usually apply well-known algebraic methods and results. Unfortunately, this does not work in the case of partially ordered sets (posets, shortly). The reason for this is that posets need not have operations and hence basic theorems of General Algebra cannot be applied. In 1990, J.~Je\v zek and R.~Quackenbush (\cite{JQ}) showed that if a poset $\mathbf P$ is up-directed or down-directed then a certain algebra with one binary operation, a so-called {\em directoid}, can be assigned to $\mathbf P$. This assignment is in general not unique, but, conversely, from every such assigned directoid, $\mathbf P$ can be reconstructed in a unique way. This fact allows to convert directed posets into algebras which bear all the information on the given poset.

In a similar way to a given poset that is both up- and down-directed one can assign an algebra with two binary operations $\sqcup$ and $\sqcap$ as shown by V.~Sn\'a\v sel (\cite S). Such an algebra is called a {\em $\lambda$-lattice}. An overview concerning results on directoids and $\lambda$-lattices can be found in our monograph \cite{CL11}.

The aforementioned approach was used also for bounded complemented posets by the authors and M.~Kola\v r\'ik in \cite{CKL}. In fact, the class of directoids assigned to such posets forms a variety of algebras. This is of great advantage since there exist many methods and results for studying varieties in General Algebra. The same machinery was used in \cite{CK} where so-called orthoposets and orthomodular posets (used in the formalization of the logic of quantum mechanics) were converted into algebras forming a variety.

The natural question arises if such a method can be applied also to pseudocomplemented, relatively pseudocomplemented and sectionally pseudocomplemented posets (the last were introduced recently by the authors and J.~Paseka in \cite{CLP21}). We solve this question by using commutative directoids and $\lambda$-lattices. We characterize the assigned algebras by means of relatively simple conditions. Unfortunately, not all of these conditions can be expressed in the form of identities or quasi-identities. Thus the corresponding classes of algebras do not form varieties or quasivarieties. On the other hand, these algebras still share nice congruence properties as we will show.

We believe that our approach can bring new insight into the study of pseudocomplemented, relatively pseudocomplemented and sectionally pseudocomplemented posets since we provide a purely algebraic description of them, thus enabling the application of algebraic tools for their investigation.

\section{Preliminaries}

Let $\mathbf P:=(P,\leq)$ be a poset, $A,B\subseteq P$ and $a,b\in P$. Then $A\leq B$ means that $x\leq y$ for all $x\in A$ and $y\in B$. Instead of $A\leq\{b\}$ and $\{a\}\leq B$ we simply write $A\leq b$ and $a\leq B$. The sets
\begin{align*}
L(A) & :=\{x\in P\mid x\leq A\}, \\
U(A) & :=\{x\in P\mid A\leq x\}
\end{align*}
are called the {\em lower} and {\em upper cone} of $A$, respectively. Instead of $L(\{a\})$ and $L(\{a,b\})$ we simply write $L(a)$ and $L(a,b)$ , respectively. In a similar way we proceed for $U$. It is easy to see that $A\subseteq L(b)$ if and only if $A\leq b$. If the infimum $\inf(a,b)$ of $a$ and $b$ exists in $\mathbf P$ then we will denote it by $a\wedge b$. The poset $\mathbf P$ is called
\begin{itemize}
\item {\em down-directed} if $L(x,y)\neq\emptyset$ for all $x,y\in P$,
\item {\em up-directed} if $U(x,y)\neq\emptyset$ for all $x,y\in P$,
\item {\em directed} if it is both down- and up-directed.
\end{itemize}
Of course, if $\mathbf P$ has a top element $1$ then it is up-direceted, and if it has a bottom element $0$ then it is down-directed.

The concept of a commutative directoid was introduced by J.~Je\v zek and R.~Quackenbush (\cite{JQ}), see also \cite{CL11} for details and elementary theory. A {\em commutative meet-directoid} is a groupoid $(D,\sqcap)$ satisfying the following identities:
\begin{itemize}
\item $x\sqcap x\approx x$,
\item $x\sqcap\big((x\sqcap y)\sqcap z\big)\approx(x\sqcap y)\sqcap z$,
\item $x\sqcap y\approx y\sqcap x$.
\end{itemize}
If $\mathbf D=(D,\sqcap)$ is a commutative meet-directoid and one defines a binary relation $\leq$ on $D$ by $x\leq y$ if $x\sqcap y=x$ then $\mathbb P(\mathbf D):=(D,\leq)$ is a down-directed poset, called the {\em poset induced} by $\mathbf D$.

Dually, one can define a {\em commutative join-directoid} $(D,\sqcup)$. If $\mathbf D=(D,\sqcup)$ is a commutative join-directoid and one defines a binary relation $\leq$ on $D$ by $x\leq y$ if $x\sqcup y=y$ then $\mathbb Q(\mathbf D):=(D,\leq)$ is an up-directed poset, called the {\em poset induced} by $\mathbf D$.

Assume $\mathbf P=(P,\leq)$ to be down-directed and define a binary operation $\sqcap$ on $P$ by the following prescription: $x\sqcap y:=\min(x,y)$ if $x$ and $y$ are comparable with each other, and $x\sqcap y=y\sqcap x$ should be an arbitrary element of $L(x,y)$ otherwise. Thus $x\sqcap y\in L(x,y)$ in any case. Then $\mathbf D:=(P,\sqcap)$ is a commutative meet-directoid, called a {\em commutative meet-directoid assigned} to $\mathbf P$, and $\mathbb P(\mathbf D)=\mathbf P$. Hence, the down-directed poset $\mathbf P$ is uniquely determined by an assigned commutative meet-directoid $\mathbf D$ in contrast to the fact that $\mathbf D$ is in general not uniquely determined by $\mathbf P$.

Dually, we define a {\em commutative join-directoid $\mathbf D=(P,\sqcup)$ assigned} to an up-directed poset $\mathbf P=(P,\leq)$. We then have $\mathbb Q(\mathbf D)=\mathbf P$.

Using an assigned commutative meet- and join-directoid, respectively, we can describe lower and upper cones of a given poset as follows.

\begin{lemma}\label{lem1}
Let $\mathbf P=(P,\leq)$ be a poset and $a,b,c\in P$. If $\mathbf P$ is down-directed and $(P,\sqcap)$ an assigned commutative meet-directoid then
\begin{enumerate}[{\rm(a)}]
\item $L(a,b)=\{(a\sqcap x)\sqcap(b\sqcap x)\mid x\in P\}$,
\item $c\in L(a,b)$ if and only if $(a\sqcap c)\sqcap(b\sqcap c)=c$.
\end{enumerate}
If $\mathbf P$ is up-directed and $(P,\sqcup)$ an assigned commutative join-directoid then
\begin{enumerate}
\item[{\rm(c)}] $U(a,b)=\{(a\sqcup x)\sqcup(b\sqcup x)\mid x\in P\}$,
\item[{\rm(d)}] $c\in U(a,b)$ if and only if $(a\sqcup c)\sqcup(b\sqcup c)=c$.
\end{enumerate}
\end{lemma}

\begin{proof}
First assume $\mathbf P$ to be down-directed and $(P,\sqcap)$ to be an assigned commutative meet-directoid.
\begin{enumerate}[(a)]
\item If $c\in L(a,b)$ then $c=c\sqcap c=(a\sqcap c)\sqcap(b\sqcap c)$. Conversely, $(a\sqcap c)\sqcap(b\sqcap c)\leq a\sqcap c\leq a$ and, analogously, $(a\sqcap c)\sqcap(b\sqcap c)\leq b$, i.e.\ $(a\sqcap c)\sqcap(b\sqcap c)\in L(a,b)$.
\item If $c\in L(a,b)$ then $(a\sqcap c)\sqcap(b\sqcap c)=c\sqcap c=c$. The converse direction follows from (i).
\end{enumerate}
The rest of the lemma follows by duality.
\end{proof}

\section{Pseudocomplemented posets}

Pseudocomplemented posets were introduced and studied by O.~Frink (\cite F), see also \cite N and \cite V for further development. Let us recall the definition.

\begin{definition}\label{def3}
A {\em pseudocomplemented poset} is an ordered quadruple $(P,\leq,{}^*,0)$ such that $(P,\leq,0)$ is a poset with bottom element $0$ and $^*$ is a unary operation on $P$ such that for all $x\in P$, $x^*$ is the greatest element $y$ of $P$ satisfying $L(x,y)=\{0\}$. The element $x^*$ is called the {\em pseudocomplement} of $x$. A {\em Stone poset} is a pseudocomplemented poset $(P,\leq,{}^*,0)$ satisfying $U(x^*,x^{**})=\{0^*\}$ for all $x\in P$.
\end{definition}

It is worth noticing that $L(x,y)=\{0\}$ means the same as $x\wedge y=0$. Hence, for every $x\in P$ we have $x\wedge x^*=0$. Moreover, every pseudocomplemented poset $(P,\leq,{}^*,0)$ has a top element, namely $0^*=1$.

\begin{example}\label{ex1}
If $(P,\leq)$ denotes the poset visualized in Figure~1:

\vspace*{-2mm}

\begin{center}
\setlength{\unitlength}{7mm}
\begin{picture}(6,10)
\put(1,3){\circle*{.3}}
\put(5,3){\circle*{.3}}
\put(1,7){\circle*{.3}}
\put(5,7){\circle*{.3}}
\put(3,9){\circle*{.3}}
\put(3,1){\circle*{.3}}
\put(1,3){\line(0,1)4}
\put(1,3){\line(1,1)4}
\put(5,3){\line(-1,1)4}
\put(5,3){\line(0,1)4}
\put(3,9){\line(-1,-1)2}
\put(3,9){\line(1,-1)2}
\put(3,1){\line(-1,1)2}
\put(3,1){\line(1,1)2}
\put(.3,2.85){$a$}
\put(5.4,2.85){$b$}
\put(.3,6.85){$c$}
\put(5.4,6.85){$d$}
\put(2.85,9.4){$1$}
\put(2.85,.25){$0$}
\put(2.2,-.75){{\rm Fig.\ 1}}
\end{picture}
\end{center}

\vspace*{4mm}

and the unary operation $^*$ on $P$ is defined by
\[
\begin{array}{c|cccccc}
 x  & 0 & a & b & c & d & 1 \\
\hline
x^* & 1 & b & a & 0 & 0 & 0
\end{array}
\]
then $(P,\leq,{}^*,0)$ is a pseudocomplemented poset which is neither a lattice nor a Stone poset since
\[
\begin{array}{c|cccccc}
   x   & 0 & a & b & c & d & 1 \\
\hline
  x^*  & 1 & b & a & 0 & 0 & 0 \\
x^{**} & 0 & a & b & 1 & 1 & 1
\end{array}
\]
and $U(a^*,a^{**})=\{c,d,1\}\neq\{1\}$.
\end{example}

\begin{example}
If $(P,\leq)$ denotes the poset visualized in Figure~2:

\vspace*{-2mm}

\begin{center}
\setlength{\unitlength}{7mm}
\begin{picture}(14,12)
\put(5,3){\circle*{.3}}
\put(9,3){\circle*{.3}}
\put(5,7){\circle*{.3}}
\put(9,7){\circle*{.3}}
\put(7,9){\circle*{.3}}
\put(7,1){\circle*{.3}}
\put(1,7){\circle*{.3}}
\put(13,7){\circle*{.3}}
\put(5,3){\line(0,1)4}
\put(5,3){\line(1,1)4}
\put(9,3){\line(-1,1)4}
\put(9,3){\line(0,1)4}
\put(7,9){\line(-1,-1)2}
\put(7,9){\line(1,-1)2}
\put(7,1){\line(-1,1)6}
\put(7,1){\line(1,1)6}
\put(7,9){\line(-3,-1)6}
\put(7,9){\line(3,-1)6}
\put(6.85,.25){$0$}
\put(4.3,2.85){$a$}
\put(9.4,2.85){$b$}
\put(4.3,6.85){$d$}
\put(9.4,6.85){$e$}
\put(6.85,9.4){$1$}
\put(.3,6.85){$c$}
\put(13.4,6.85){$f$}
\put(6.2,-.75){{\rm Fig.\ 2}}
\end{picture}
\end{center}

\vspace*{4mm}

and the unary operation $^*$ on $P$ is defined by
\[
\begin{array}{c|cccccccc}
 x  & 0 & a & b & c & d & e & f & 1 \\
\hline
x^* & 1 & f & c & f & 0 & 0 & c & 0
\end{array}
\]
then $(P,\leq,{}^*,0)$ is not a lattice, but a Stone poset since
\[
\begin{array}{c|cccccccc}
   x   & 0 & a & b & c & d & e & f & 1 \\
\hline
  x^*  & 1 & f & c & f & 0 & 0 & c & 0 \\
x^{**} & 0 & c & f & c & 1 & 1 & f & 1
\end{array}
\]
and $U(x^*,x^{**})=\{1\}$ for all $x\in P$.
\end{example}

Now we show how a pseudocomplemented poset can be characterized by an assigned commutative meet-directoid equipped with a unary operation $^*$ and a nullary operation $0$.

\begin{theorem}\label{th1}
Let $(P,\leq)$ be a down-directed poset, $^*$ a unary operation on $P$, $0\in P$ and $(P,\sqcap)$ a commutative meet-directoid assigned to $(P,\leq)$. Then $\mathbf P=(P,\leq,{}^*,0)$ is a pseudocomplemented poset if and only if $\mathbf A=(P,\sqcap,{}^*,0)$ satisfies the following conditions:
\begin{enumerate}[{\rm(i)}]
\item $0\sqcap x\approx0$,
\item $(x\sqcap y)\sqcap(x^*\sqcap y)\approx0$,
\item $(x\sqcap z)\sqcap(y\sqcap z)=0\,\forall z\in P\Rightarrow y\sqcap x^*=y$.
\end{enumerate}
In this case we call $\mathbf A$ an {\em algebra assigned} to $\mathbf P$.
\end{theorem}

\begin{proof}
$\mathbf P$ is a pseudocomplemented poset if and only if the following hold:
\begin{enumerate}[(i')]
\item $0\leq x$,
\item $L(x,x^*)=\{0\}$,
\item $L(x,y)=\{0\}\Rightarrow y\leq x^*$.
\end{enumerate}
Now $\mathbf P$ satisfies (i') if and only if $\mathbf A$ satisfies (i). Because of Lemma~\ref{lem1} (a), $\mathbf P$ satisfies (ii') if and only if $\mathbf A$ satisfies (ii), and $\mathbf P$ satisfies (iii') if and only if $\mathbf A$ satisfies (iii).
\end{proof}

The concept of a $\lambda$-lattice was introduced by V.~Sn\'a\v sel (\cite S), see also \cite{CL11}. A {\em $\lambda$-lattice} is an algebra $(L,\sqcup,\sqcap)$ of type $(2,2)$ satisfying the following identities:
\begin{itemize}
\item $x\sqcup y\approx y\sqcup x$, $x\sqcap y\approx y\sqcap x$,
\item $x\sqcup\big((x\sqcup y)\sqcup z\big)\approx(x\sqcup y)\sqcup z$, $x\sqcap\big((x\sqcap y)\sqcap z\big)\approx(x\sqcap y)\sqcap z$,
\item $(x\sqcup y)\sqcap x\approx x$, $(x\sqcap y)\sqcup x\approx x$.
\end{itemize}
Let $\mathbf L=(L,\sqcup,\sqcap)$ be a $\lambda$-lattice. Then $(L,\sqcup)$ and $(L,\sqcap)$ are commutative join- and meet-directoids, respectively, and $x\sqcap y=x$ if and only if $x\sqcup y=y$. If one defines a binary relation $\leq$ on $L$ by $x\leq y$ if $x\sqcap   y=x$ (or, equivalently, $x\sqcup y=y$) then $\mathbb R(\mathbf L):=(L,\leq)$ is a directed poset, called the {\em poset induced} by $\mathbf L$.

Let $\mathbf P=(P,\leq)$ be a directed poset and define binary operations $\sqcup$ and $\sqcap$ on $P$ by the following prescription: $x\sqcup y:=\max(x,y)$ and $x\sqcap y:=\min(x,y)$ if $x$ and $y$ are comparable with each other, and $x\sqcup y=y\sqcup x$ and $x\sqcap y=y\sqcap x$ should be arbitrary elements of $U(x,y)$ and $L(x,y)$, respectively, otherwise. Then $\mathbf L:=(P,\sqcup,\sqcap)$ is a $\lambda$-lattice, called a {\em $\lambda$-lattice assigned} to $\mathbf P$, and $\mathbb R(\mathbf L)=\mathbf P$. Hence, a given directed poset $\mathbf P$ is uniquely determined by an assigned $\lambda$-lattice $\mathbf L$ contrary to the fact that $\mathbf L$ may be assigned to $\mathbf P$ in a non-unique way. It is easy to see that $(P,\sqcup)$ and $(P,\sqcap)$ are commutative join- and meet-directoids assigned to $\mathbf P$, respectively.

As remarked above, every Stone poset $\mathbf P$ has the top element $0^*$, thus it is also up-directed and hence directed. Due to this, we can assign to $\mathbf P$ a $\lambda$-lattice.

\begin{theorem}
Let $(P,\leq)$ be a directed poset, $^*$ a unary operation on $P$, $0\in P$ and $(P,\sqcup,\sqcap)$ a $\lambda$-lattice assigned to $(P,\leq)$. Then $\mathbf P=(P,\leq,{}^*,0)$ is a Stone poset if and only if $\mathbf A=(P,\sqcup,\sqcap,{}^*,0)$ satisfies {\rm(i)} -- {\rm(iii)} of Theorem~\ref{th1} as well as
\begin{enumerate}
\item[{\rm(iv)}] $(x^*\sqcup y)\sqcup(x^{**}\sqcup y)\approx0^*$.
\end{enumerate}
In this case we call $\mathbf A$ an {\em algebra assigned} to $\mathbf P$.
\end{theorem}

\begin{proof}
$\mathbf P$ is a Stone poset if and only if it is a pseudocomplemented poset satisfying
\begin{enumerate}
\item[{\rm(iv')}] $U(x^*,x^{**})=\{0^*\}$.
\end{enumerate}
Because of Lemma~\ref{lem1} (c), $\mathbf P$ satisfies (iv') if and only if $\mathbf A$ satisfies (iv). The rest follows from Theorem~\ref{th1}.
\end{proof}

Let us recall that a {\em poset} $(P,\leq)$ is called {\em distributive} if it satisfies the equality
\begin{equation}\label{equ3}
U\big(L(x,y),z\big)=UL\big(U(x,z),U(y,z)\big)
\end{equation}
or, equivalently,
\begin{equation*}
L\big(U(x,y),z\big)=LU\big(L(x,z),L(y,z)\big)
\end{equation*}
for all $x,y,z\in P$. We are going to show that distributive pseudocomplemented posets satisfying $U(x,x^*)=\{0^*\}$ can be characterized by means of equalities only.

\begin{theorem}\label{th6}
Let $(P,\leq,{}^*,0,1)$ be a bounded distributive poset with a unary operation $^*$ satisfying $U(x,x^*)=\{1\}$ for all $x\in P$. Then $\mathbf P=(P,\leq,{}^*,0)$ is pseudocomplemented if and only if it satisfies the following equalities for all $x,y\in P$:
\begin{enumerate}[{\rm(i)}]
\item $L(x,x^*)=\{0\}$,
\item $U\big(x^*,L(x,y)\big)=U(x^*,y)$.
\end{enumerate}
\end{theorem}

\begin{proof}
Let $a,b\in P$. If $\mathbf P$ is pseudocomplemented then (i) follows from Definition~\ref{def3} and (ii) follows from
\[
U\big(a^*,L(a,b)\big)=UL\big(U(a^*,a),U(a^*,b)\big)=UL\big(1,U(a^*,b)\big)=ULU(a^*,b)=U(a^*,b)
\]
by using distributivity of $(P,\leq)$ and $U(a,a^*)=\{1\}$. If, conversely, $\mathbf P$ satisfies (i) and (ii) Then $L(a,a^*)=\{0\}$ by (i) and if $L(a,b)=\{0\}$ then
\[
a^*\in U(a^*)=U(a^*,0)=U\big(a^*,L(a,b)\big)=U(a^*,b)\subseteq U(b)
\]
by (ii) whence $b\leq a^*$.
\end{proof}

An example of a poset satisfying the assumptions of Theorem~\ref{th6} is the following.

\begin{example}
If $(P,\leq)$ denotes the poset visualized in Figure~3:

\vspace*{-2mm}

\begin{center}
\setlength{\unitlength}{7mm}
\begin{picture}(8,12)
\put(4,1){\circle*{.3}}
\put(1,4){\circle*{.3}}
\put(3,4){\circle*{.3}}
\put(5,4){\circle*{.3}}
\put(7,4){\circle*{.3}}
\put(1,6){\circle*{.3}}
\put(7,6){\circle*{.3}}
\put(1,8){\circle*{.3}}
\put(3,8){\circle*{.3}}
\put(5,8){\circle*{.3}}
\put(7,8){\circle*{.3}}
\put(4,11){\circle*{.3}}
\put(4,1){\line(-1,1)3}
\put(4,1){\line(-1,3)1}
\put(4,1){\line(1,3)1}
\put(4,1){\line(1,1)3}
\put(4,11){\line(-1,-1)3}
\put(4,11){\line(-1,-3)1}
\put(4,11){\line(1,-3)1}
\put(4,11){\line(1,-1)3}
\put(1,4){\line(0,1)4}
\put(1,4){\line(1,1)4}
\put(3,4){\line(-1,1)2}
\put(3,4){\line(1,1)4}
\put(5,4){\line(-1,1)4}
\put(5,4){\line(1,1)2}
\put(7,4){\line(-1,1)4}
\put(7,4){\line(0,1)4}
\put(1,6){\line(1,1)2}
\put(7,6){\line(-1,1)2}
\put(3.85,.25){$0$}
\put(.3,3.85){$a$}
\put(2.3,3.85){$b$}
\put(5.4,3.85){$c$}
\put(7.4,3.85){$d$}
\put(.3,5.85){$e$}
\put(7.4,5.85){$f$}
\put(.3,7.85){$g$}
\put(2.3,7.85){$h$}
\put(5.4,7.85){$i$}
\put(7.4,7.85){$j$}
\put(3.85,11.4){$1$}
\put(3.2,-.75){{\rm Fig.\ 3}}
\end{picture}
\end{center}

\vspace*{4mm}

and the unary operation $^*$ on $P$ is defined by
\[
\begin{array}{c|cccccccccccc}
 x  & 0 & a & b & c & d & e & f & g & h & i & j & 1 \\
\hline
x^* & 1 & j & i & h & g & f & e & d & c & b & a & 0
\end{array}
\]
then $(P,\leq,{}^*,0)$ is a distributive Stone poset {\rm(}which is not a lattice{\rm)} satisfying $U(x,x^*)=\{1\}$ for all $x\in P$ since
\[
\begin{array}{c|cccccccccccc}
   x   & 0 & a & b & c & d & e & f & g & h & i & j & 1 \\
\hline
  x^*  & 1 & j & i & h & g & f & e & d & c & b & a & 0 \\
x^{**} & 0 & a & b & c & d & e & f & g & h & i & j & 1
\end{array}
\]
\end{example}

\section{Relatively pseudocomplemented posets}

Relatively pseudocomplemented posets were studied by numerous authors from different points of view, e.g.\ as a base of some non-classical logics where the relative pseudocomplementation is considered as the connective implication (see e.g.\ \cite K) or as an approach to Hilbert algebras (see e.g.\ \cite{CLP20} and \cite R), or for purely algebraic reasons, see e.g.\ \cite{CL18}. Recall the following definition.

\begin{definition}\label{def1}
A {\em relatively pseudocomplemented poset} is an ordered quadruple $(P,\leq,*,1)$ such that $(P,\leq)$ is a poset and $*$ is a binary operation on $P$ such that for all $x,y\in P$, $x*y$ is the greatest element $z$ of $P$ satisfying $L(x,z)\subseteq L(y)$ and $1$ denotes the top element of $(P,\leq)$. {\rm(}The existence of such an element follows from the fact that $x*x$ is the greatest element $y$ of $(P,\leq)$ satisfying $L(x,y)\subseteq L(x)$.{\rm)} The element $x*y$ is called the {\em relative pseudocomplement} of $x$ with respect to $y$.
\end{definition}

It is clear that a relatively pseudocomplemented poset having a bottom element $0$ is pseudocomplemented since $x^*=x*0$. As it was shown by the authors in \cite{CL18}, every relatively pseudocomplemented poset is distributive.

It is well-known that relatively pseudocomplemented semilattices or lattices play important roles in the axiomatization of intuitionistic logics. They are known under the names {\em Brouwerian {\rm(}semi-{\rm)}lattices} or {\em Heyting algebras}, see e.g.\ \cite{JQ}. However, also relatively pseudocomplemented posets can be recognized as a formalization of certain logics of this sort, where conjunction is unsharp in the following sense. Directly by Definition~\ref{def1} we have
\[
L(a,c)\leq b\text{ if and only if }c\leq a*b.
\]
The binary operation $*$ can be considered as the logical connective of implication. From above we have
\[
L(a,a*b)\leq b.
\]
This can be transferred to the language of propositional calculus as follows: If we know truth values of $a$ and $a*b$ then the truth value of $b$ cannot be less than that of $a$ and $a*b$. This expresses an unsharp version of the derivation rule {\em Modus Ponens}.

A typical example of a relatively pseudocomplemented poset is the following.

\begin{example}
If $(P,\leq)$ denotes the poset from Example~\ref{ex1} and the binary operation $*$ on $P$ is defined by
\[
\begin{array}{c|cccccc}
* & 0 & a & b & c & d & 1 \\
\hline
0 & 1 & 1 & 1 & 1 & 1 & 1 \\
a & b & 1 & b & 1 & 1 & 1 \\
b & a & a & 1 & 1 & 1 & 1 \\
c & 0 & a & b & 1 & d & 1 \\
d & 0 & a & b & c & 1 & 1 \\
1 & 0 & a & b & c & d & 1
\end{array}
\]
then $(P,\leq,*,1)$ is a relatively pseudocomplemented poset which is not a lattice.
\end{example}

Also down-directed relatively pseudocomplemented posets can be characterized by means of assigned directoids equipped with a binary operation.

\begin{theorem}\label{th3}
Let $(P,\leq)$ be a down-directed poset, $*$ a binary operation on $P$, $1\in P$ and $(P,\sqcap)$ a commutative meet-directoid   corresponding to $(P,\leq)$. Then $\mathbf P=(P,\leq,*,1)$ is a relatively pseudocomplemented poset if and only if $\mathbf A=(P,\sqcap,*,1)$ satisfies the following conditions:
\begin{enumerate}[{\rm(i)}]
\item $x\sqcap1\approx x$,
\item $\Big((x\sqcap z)\sqcap\big((x*y)\sqcap z\big)\Big)\sqcap y\approx(x\sqcap z)\sqcap\big((x*y)\sqcap z\big)$,
\item $\big((x\sqcap t)\sqcap(z\sqcap t)\big)\sqcap y=(x\sqcap t)\sqcap(z\sqcap t)\,\forall t\in P\Rightarrow z\sqcap(x*y)=z$.
\end{enumerate}
In this case we call $\mathbf A$ an {\em algebra assigned} to $\mathbf P$.
\end{theorem}

\begin{proof}
$\mathbf P$ is a relatively pseudocomplemented poset if and only if the following hold:
\begin{enumerate}
\item[(i')] $x\leq1$,
\item[(ii')] $L(x,x*y)\leq y$,
\item[(iii')] $L(x,z)\leq y\Rightarrow z\leq x*y$.
\end{enumerate}
Obviously, $\mathbf P$ satisfies (i') if and only if $\mathbf A$ satisfies (i). Because of Lemma~\ref{lem1} (a), $\mathbf P$ satisfies (ii') if and only if $\mathbf A$ satisfies (ii), and $\mathbf P$ satisfies (iii') if and only if $\mathbf A$ satisfies (iii).
\end{proof}

As shown above, a relatively pseudocomplemented poset $\mathbf P$ is determined by its assigned algebra $\mathbf A$ which is a commutative meet-directoid with constant $1$ and equipped with a binary oparation $*$. Hence, this algebra shares all the properties of $\mathbf P$, but expressed in the language of $\mathbf A$. In the following we show how some properties of $\mathbf A$ can be derived directly from conditions (i) -- (iii) of Theorem~\ref{th3}.

\begin{theorem}\label{th4}
Let $\mathbf A=(P,\sqcap,*,1)$ be an algebra assigned to a relatively pseudocomplemented poset $(P,\leq,*,1)$. Then it satisfies the following identities:
\begin{enumerate}[{\rm(a)}]
\item $x*x\approx1$,
\item $1*x\approx x$,
\item $x\sqcap\big((x*y)*y\big)\approx x$.
\end{enumerate}
\end{theorem}

\begin{proof}
\
\begin{enumerate}[(a)]
\item We have
\[
(x\sqcap t)\sqcap(1\sqcap t)\leq x\sqcap t\leq x\forall t\in P
\]
and hence
\[
\big((x\sqcap t)\sqcap(1\sqcap t)\big)\sqcap x=(x\sqcap t)\sqcap(1\sqcap t)\forall t\in P
\]
which implies
\[
x*x=1\sqcap(x*x)=1
\]
by (i) and by (iii) when replacing $y$ and $z$ by $x$ and $1$, respectively.
\item We have
\[
1*x=1\sqcap(1*x)=(1\sqcap1)\sqcap\big((1*x)\sqcap1\big)=\Big((1\sqcap1)\sqcap\big((1*x)\sqcap1\big)\Big)\sqcap x\leq x
\]
by (i) and by (ii) when replacing $x$, $y$ and $z$ by $1$, $x$ and $1$, respectively. Conversely, we have
\begin{align*}
\big((1\sqcap t)\sqcap(x\sqcap t)\big)\sqcap x & =\big(t\sqcap(x\sqcap t)\big)\sqcap x=(x\sqcap t)\sqcap x=x\sqcap t=t\sqcap(x\sqcap t)= \\
                                               & =(1\sqcap t)\sqcap(x\sqcap t)\forall t\in P
\end{align*}
and hence
\[
x=x\sqcap(1*x)\leq1*x
\]
by (iii) when replacing $x$, $y$ and $z$ by $1$, $x$ and $x$, respectively. Altogether, we obtain $1*x\approx x$.
\item We have
\[
\Big((x\sqcap t)\sqcap\big((x*y)\sqcap t\big)\Big)\sqcap y=(x\sqcap t)\sqcap\big((x*y)\sqcap t\big)\forall t\in P
\]
by (ii) when replacing $z$ by $t$ and hence
\[
\Big(\big((x*y)\sqcap t\big)\sqcap(x\sqcap t)\Big)\sqcap y=\big((x*y)\sqcap t\big)\sqcap(x\sqcap t)\forall t\in P
\]
which implies
\[
x\sqcap\big((x*y)*y\big)=x
\]
by (iii) when replacing $x$ and $z$ by $x*y$ and $x$, respectively,
\end{enumerate}
\end{proof}

\section{Sectionally pseudocomplemented posets}

Sectionally pseudocomplemented posets were recently introduced by the authors and J.~Paseka in \cite{CLP21}.

\begin{definition}\label{def2}
A {\em sectionally pseudocomplemented poset} is an ordered triple $(P,\leq,\circ)$ such that $(P,\leq)$ is a poset and $\circ$ is a binary operation on $P$ such that for all $x,y\in P$, $x\circ y$ is the greatest element $z$ of $P$ satisfying $L(U(x,y),z)=L(y)$. The element $x\circ y$ is called the {\em sectional pseudocomplement} of $x$ with respect to $y$. A {\em sectionally pseudocomplemented poset with $1$} is an ordered quadruple $(P,\leq,\circ,1)$ such that $(P,\leq,\circ)$ is a sectionally pseudocomplemented poset and $1$ is the top element of $(P,\leq)$. A {\em strongly sectionally pseudocomplemented poset} {\rm(}see {\rm\cite{CLP21})} is a sectionally pseudocomplemented poset $(P,\leq,\circ,1)$ with $1$ satisfying $x\leq(x\circ y)\circ y$ for all $x,y\in P$.
\end{definition}

If a sectionally pseudocomplemented poset $(P,\leq,\circ)$ has a top element $1$ then $1=x\circ x$ for all $x\in P$, but not every sectionally pseudocomplemented poset has a top element, see the following example.

\begin{example}
If $(P,\leq)$ denotes the poset visualized in Figure~4:

\vspace*{-2mm}

\begin{center}
\setlength{\unitlength}{7mm}
\begin{picture}(4,6)
\put(0,1){\circle*{.3}}
\put(4,1){\circle*{.3}}
\put(0,5){\circle*{.3}}
\put(4,5){\circle*{.3}}
\put(0,1){\line(0,1)4}
\put(0,1){\line(1,1)4}
\put(4,1){\line(-1,1)4}
\put(4,1){\line(0,1)4}
\put(-.15,.25){$a$}
\put(3.85,.25){$b$}
\put(-.15,5.4){$c$}
\put(3.85,5.4){$d$}
\put(1.2,-.75){{\rm Fig.\ 4}}
\end{picture}
\end{center}

\vspace*{4mm}

and the binary operation $\circ$ on $P$ is defined by $x\circ y:=y$ for all $x,y\in P$ then $(P,\leq,\circ)$ is a sectionally pseudocomplemented poset which is neither a lattice nor relatively pseudocomplemented nor strongly sectionally pseudocomplemented since it has no top element.
\end{example}

\begin{example}
If $\mathbf P=(P,\leq)$ denotes the poset visualized in Figure~5:

\vspace*{-2mm}

\begin{center}
\setlength{\unitlength}{7mm}
\begin{picture}(6,10)
\put(1,3){\circle*{.3}}
\put(2,2){\circle*{.3}}
\put(5,3){\circle*{.3}}
\put(1,7){\circle*{.3}}
\put(5,7){\circle*{.3}}
\put(3,9){\circle*{.3}}
\put(3,1){\circle*{.3}}
\put(1,3){\line(0,1)4}
\put(1,3){\line(1,1)4}
\put(5,3){\line(-1,1)4}
\put(5,3){\line(0,1)4}
\put(3,9){\line(-1,-1)2}
\put(3,9){\line(1,-1)2}
\put(3,1){\line(-1,1)2}
\put(3,1){\line(1,1)2}
\put(.3,2.85){$b$}
\put(1.3,1.85){$a$}
\put(5.4,2.85){$c$}
\put(.3,6.85){$d$}
\put(5.4,6.85){$e$}
\put(2.85,9.4){$1$}
\put(2.85,.25){$0$}
\put(2.2,-.75){{\rm Fig.\ 5}}
\end{picture}
\end{center}

\vspace*{4mm}

and the binary operation $\circ$ on $P$ is defined by
\[
\begin{array}{c|ccccccc}
\circ & 0 & a & b & c & d & e & 1 \\
\hline
  0   & 1 & 1 & 1 & 1 & 1 & 1 & 1 \\
	a   & c & 1 & 1 & c & 1 & 1 & 1 \\
	b   & c & a & 1 & c & 1 & 1 & 1 \\
	c   & b & a & b & 1 & 1 & 1 & 1 \\
	d   & 0 & a & b & c & 1 & e & 1 \\
	e   & 0 & a & b & c & d & 1 & 1 \\
	1   & 0 & a & b & c & d & e & 1
\end{array}
\]
then $(P,\leq,\circ,1)$ is a strongly sectionally pseudocomplemented poset which is neither a lattice nor relatively pseudocomplemented since the relative pseudocomplement of $b$ with respect to $a$ does not exist. Since
\[
L\big(U(a,c),b\big)=LU(b)=L(b)\neq L(a)=LUL(a)=LU\big(L(a,b),L(c,b)\big),
\]
$\mathbf P$ is not distributive.
\end{example}

An example of a sectionally pseudocomplemented poset (having a top element $1$) which is not strongly sectionally pseudocomplemented is given in \cite{CLP21}.

It is well-known that every relatively pseudocomplemented lattice is distributive. Sectional pseudocomplementation was introduced by the first author (\cite C) in order to extend the concept of relative pseudocomplementation to non-distributive lattices. It was shown that sectionally pseudocomplemented lattices form a variety and, considered as posets, they are strongly sectionally pseudocomplemented. It is a natural question if also these posets can be considered as a formalization of certain unsharp propositional logic. We can consider the binary operation $\circ$ of sectional pseudocomplementation as the logical connective of implication. Having two propositions with truth values $a$ and $b$ such that $b\leq a$, the formula
\[
L(U(a,b),c)=L(b)
\]
can be rewritten as
\[
a\wedge c=b.
\]
Hence, by Definition~\ref{def2}, for $b\leq a$,
\[
a\wedge c=b\text{ if and only if }c\leq a\circ b.
\]
Therefore, for $b\leq a$ we have a form of {\em adjointness} of operations $\wedge$ and $\circ$. Moreover, we have
\[
a\wedge(a\circ b)=b.
\]
Thus the truth value of $b$ is the same as truth value of conjunction $a\wedge(a\circ b)$ provided $b\leq a$. This is again a version of {\em Modus Ponens} restricted to elements $a,b$ with $b\leq a$.

Also directed sectionally pseudocomplemented posets can be characterized by assigned $\lambda$-lattices with a binary operation as follows.

\begin{theorem}\label{th2}
Let $(P,\leq)$ be a directed poset, $\circ$ a binary operation on $P$ and $(P,\sqcup,\sqcap)$ a $\lambda$-lattice corresponding to $(P,\leq)$. Then the following hold:
\begin{enumerate}[{\rm(a)}]
\item $\mathbf P=(P,\leq,\circ)$ is a sectionally pseudocomplemented poset if and only if $\mathbf A=(P,\sqcup,\sqcap,\circ)$ satisfies the following conditions:
\begin{enumerate}[{\rm(i)}]
\item $y\sqcap(x\circ y)\approx y$,
\item $\Big(\big((x\sqcup t)\sqcup(y\sqcup t)\big)\sqcap z\Big)\sqcap\big((x\circ y)\sqcap z\big)=z\,\forall t\in P\Rightarrow z\sqcap y=z$,
\item $\bigg(\Big(\big((x\sqcup t)\sqcup(y\sqcup t)\big)\sqcap s\Big)\sqcap(z\sqcap s)=s\,\forall t\in P\Leftrightarrow s\sqcap y=s\bigg)\,\forall s\in P\Rightarrow z\sqcap(x\circ y)=z$.
\end{enumerate}
\item If $1\in P$ then $\mathbf P=(P,\leq,\circ,1)$ is a sectionally pseudocomplemented poset with $1$ if and only if $\mathbf A=(P,\sqcup,\sqcap,\circ,1)$ satisfies {\rm(i)} -- {\rm(iii)} as well as
\begin{enumerate}
\item[{\rm(iv)}] $x\sqcap1\approx x$.
\end{enumerate}
\end{enumerate}
In both cases we call $\mathbf A$ an {\em algebra assigned} to $\mathbf P$.
\end{theorem}

\begin{proof}
\
\begin{enumerate}[(a)]
\item $\mathbf P$ is a sectionally pseudocomplemented poset if and only if the following hold:
\begin{enumerate}
\item[($\alpha$)] $L\big(U(x,y),x\circ y\big)=L(y)$,
\item[(iii')] $L\big(U(x,y),z\big)=L(y)\Rightarrow z\leq x\circ y$.
\end{enumerate}
Now $\mathbf P$ satisfies ($\alpha$) and (iii') if and only if it satisfies (i') -- (iii') where (i') and (ii') denote the following conditions:
\begin{enumerate}[(i')]
\item $y\leq x\circ y$,
\item $L\big(U(x,y),x\circ y\big)\leq y$.
\end{enumerate}
This can be seen as follows: If $\mathbf P$ satisfies ($\alpha$) and (iii') then it satisfies (ii') and because of $L\big(U(x,y),y\big)=L(y)$ and (iii') it satisfies (i'). If, conversely, $\mathbf P$ satisfies (i') -- (iii') then because of
\[
L(y)=L\big(U(x,y),y\big)\subseteq L\big(U(x,y),x\circ y\big)\subseteq L(y)
\]
it satisfies ($\alpha$). This shows that $\mathbf P$ is a sectionally pseudocomplemented poset if and only it satisfies (i') -- (iii'). Now $\mathbf P$ satisfies (i') if and only if $\mathbf A$ satisfies (i). Because of Lemma~\ref{lem1} (c) and (b), $\mathbf P$ satisfies (ii') if and only if $\mathbf A$ satisfies (ii), and $\mathbf P$ satisfies (iii') if and only if $\mathbf A$ satisfies (iii).
\item $\mathbf P$ is a sectionally pseudocomplemented poset with $1$ if and only if $(P,\leq,\circ)$ is a sectionally pseudocomplemented poset satisfying
\begin{enumerate}
\item[(iv')] $x\leq1$.
\end{enumerate}
Now $\mathbf P$ satisfies (iv') if and only if $\mathbf A$ satisfies (iv). The rest follows from (a).
\end{enumerate}
\end{proof}

In the next theorem we show how some properties of algebras assigned to sectionally pseudocomplemented posets with $1$ can be derived from conditions (i) -- (iv) of Theorem~\ref{th2}. As we remarked above, in a $\lambda$-lattice $(A,\sqcup,\sqcap)$ the identities $x\sqcap1\approx x$ and $x\sqcup1\approx1$ are equivalent.

\begin{theorem}\label{th5}
Let $\mathbf A=(P,\sqcup,\sqcap,\circ,1)$ be an algebra assigned to a sectionally pseudocomplemented poset $(P,\leq,\circ,1)$ with $1$. Then it satisfies the following identities:
\begin{enumerate}[{\rm(a)}]
\item $x\circ x\approx1$,
\item $1\circ x\approx x$.
\end{enumerate}
\end{theorem}

\begin{proof}
\
\begin{enumerate}[(a)]
\item According to (iv) we have
\[
\Big(\big((x\sqcup t)\sqcup(x\sqcup t)\big)\sqcap s\Big)\sqcap(1\sqcap s)=\big((x\sqcup t)\sqcap s\big)\sqcap s=(x\sqcup t)\sqcap s
\]
and hence the following are equivalent:
\begin{align*}
\Big(\big((x\sqcup t)\sqcup(x\sqcup t)\big)\sqcap s\Big)\sqcap(1\sqcap s) & =s\forall t\in P, \\
                                                      (x\sqcup t)\sqcap s & =s\forall t\in P, \\
                                                                        s & \leq x\sqcup t\forall t\in P, \\
                                                                        s & \leq x, \\
                                                                s\sqcap x & =s
\end{align*}
which implies
\[
x\circ x=1\sqcap(x\circ x)=1
\]
by (iv) and by (iii) when replacing $y$ and $z$ by $x$ and $1$, respectively.
\item According to (iv) we have
\begin{align*}
& \Big(\big((1\sqcup t)\sqcup(x\sqcup t)\big)\sqcap(1\circ x)\Big)\sqcap\big((1\circ x)\sqcap(1\circ x)\big)= \\
& =\Big(\big(1\sqcup(x\sqcup t)\big)\sqcap(1\circ x)\Big)\sqcap(1\circ x)=\big(1\sqcap(1\circ x)\big)\sqcap(1\circ x)=(1\circ x)\sqcap(1\circ x)= \\
& =1\circ x\forall t\in P
\end{align*}
and hence
\begin{equation}\label{equ1}
(1\circ x)\sqcap x=1\circ x
\end{equation}
by (ii) when replacing $x$, $y$ and $z$ by $1$, $x$ and $1\circ x$, respectively. Moreover, according to (iv) we have
\begin{align*}
\Big(\big((1\sqcup t)\sqcup(x\sqcup t)\big)\sqcap s\Big)\sqcap(x\sqcap s) & =\Big(\big(1\sqcup(x\sqcup t)\big)\sqcap s\Big)\sqcap(x\sqcap s)= \\
                                                                          & =(1\sqcap s)\sqcap(x\sqcap s)=s\sqcap(x\sqcap s)=x\sqcap s=s\sqcap x
\end{align*}
and hence the following are equivalent:
\begin{align*}
\Big(\big((1\sqcup t)\sqcup(x\sqcup t)\big)\sqcap s\Big)\sqcap(x\sqcap s) & =s\forall t\in P, \\
                                                                s\sqcap x & =s.
\end{align*}
This implies
\begin{equation}\label{equ2}
x\sqcap(1\circ x)=x
\end{equation}
by (iii) when replacing $x$, $y$ and $z$ by $1$, $x$ and $x$, respectively. From (\ref{equ1}) and (\ref{equ2}) we obtain (b).
\end{enumerate}
\end{proof}
 
Analogously to the characterizations in Theorem~\ref{th2} we can characterize also strongly sectionally pseudocomplemented posets.

\begin{theorem}
Let $(P,\leq)$ be a directed poset, $\circ$ a binary operation on $P$, $1\in P$ and $(P,\sqcup,\sqcap)$ a $\lambda$-lattice corresponding to $(P,\leq)$. Then $\mathbf P=(P,\leq,\circ,1)$ is a strongly sectionally pseudocomplemented poset if and only if $\mathbf A=(P,\sqcup,\sqcap,\circ,1)$ satisfies {\rm(i)} -- {\rm(iv)} of Theorem~\ref{th2} as well as
\begin{enumerate}
\item[{\rm(v)}] $x\sqcap\big((x\circ y)\circ y\big)\approx x$.
\end{enumerate}
In this case we call $\mathbf A$ an {\em algebra assigned} to $\mathbf P$.
\end{theorem}

\begin{proof}
$\mathbf P$ is a strongly sectionally pseudocomplemented poset if and only if $(P,\leq,\circ)$ is a sectionally pseudocomplemented poset with $1$ satisfying
\begin{enumerate}
\item[(v')] $x\leq(x\circ y)\circ y$.
\end{enumerate}
Now $\mathbf P$ satisfies (v') if and only if $\mathbf A$ satisfies (v). The rest follows from Theorem~\ref{th2}.
\end{proof}

\section{Congruence properties}

Next we consider congruence properties of algebras. For the convenience of the reader we recall the corresponding definitions.

An algebra $\mathbf A$ is called
\begin{itemize}
\item {\em congruence permutable} if $\Theta\circ\Phi=\Phi\circ\Theta$ for all $\Theta,\Phi\in\Con\mathbf A$,
\item {\em congruence distributive} if the congruence lattice of $\mathbf A$ is distributive,
\item {\em arithmetical} if it is both congruence permutable and congruence distributive,
\item {\em weakly regular} (with respect to an equationally definable constant $1$) if $\Theta,\Phi\in\Con\mathbf A$ and $[1]\Theta=[1]\Phi$ imply $\Theta=\Phi$.
\end{itemize}

The following results are well-known (see e.g.\ \cite{CEL}) :

Let $\mathbf A$ be an algebra.
\begin{itemize}
\item If there exists a ternary term $p$ satisfying the identities $p(x,x,y)\approx p(y,x,x)\approx y$ then $\mathbf A$ is congruence permutable.
\item If there exists a ternary term $m$ satisfying the identities $m(x,x,y)\approx m(x,y,x)\approx m(y,x,x)\approx x$ then $\mathbf A$ is congruence distributive.
\item If $\mathbf A$ has an equationally definable constant $1$ and there exists a positive integer $n$, binary terms $t_1,\ldots,t_n$ and quaternary terms $s_1,\ldots,s_n$ satisfying the identities
\begin{align*}
& t_1(x,x)\approx\cdots\approx t_n(x,x)\approx1, \\
& s_1\big(t_1(x,y),1,x,y\big)\approx x, \\
& s_i\big(1,t_i(x,y),x,y\big)\approx s_{i+1}\big(t_{i+1}(x,y),1,x,y\big)\text{ for }i=1,\ldots,n-1, \\
& s_n\big(1,t_n(x,y),x,y\big)\approx y
\end{align*}
then $\mathbf A$ is weakly regular.
\end{itemize}
The terms $p$ and $m$ are called {\em Maltsev term} and {\em majority term}, respectively.

Now we can prove

\begin{theorem}
\
\begin{enumerate}[{\rm(i)}]
\item Let $(P,\leq,{}^*,0)$ be a Stone poset and $\mathbf A=(P,\sqcup,\sqcap,{}^*,0)$ be an assigned algebra. Then $\mathbf A$ is congruence distributive.
\item Let $(P,\leq,*,1)$ be a down-directed relatively pseudocomplemented poset and $\mathbf A=(P,\sqcap,*,1)$ an assigned algebra. Then $\mathbf A$ is congruence permutable and weakly regular.
\item Let $(P,\leq,\circ)$ be a directed sectionally pseudocomplemented poset and $\mathbf A=(P,\sqcup,\sqcap,\circ)$ an assigned algebra. Then $\mathbf A$ is congruence distributive.
\item Let $(P,\leq,\circ,1)$ be a down-directed sectionally pseudocomplemented poset with $1$ and $\mathbf A=(P,\sqcup,\sqcap,\circ,1)$ an assigned algebra. Then $\mathbf A$ is congruence distributive and weakly regular.
\item Let $(P,\leq,\circ,1)$ be a down-directed strongly sectionally pseudocomplemented poset and $\mathbf A=(P,\sqcup,\sqcap,\circ,1)$ an assigned algebra. Then $\mathbf A$ is arithmetical and weakly regular.
\end{enumerate}
\end{theorem}

\begin{proof}
\
\begin{enumerate}[(i)]
\item If
\[
m(x,y,z):=\big((x\sqcup y)\sqcap(y\sqcup z)\big)\sqcap(z\sqcup x)
\]
then
\begin{align*}
m(x,x,y) & \approx\big((x\sqcup x)\sqcap(x\sqcup y)\big)\sqcap(y\sqcup x)\approx\big(x\sqcap(x\sqcup y)\big)\sqcap(y\sqcup x)\approx x\sqcap(y\sqcup x)\approx \\
         & \approx x, \\
m(x,y,x) & \approx\big((x\sqcup y)\sqcap(y\sqcup x)\big)\sqcap(x\sqcup x)\approx(x\sqcup y)\sqcap x\approx x, \\
m(y,x,x) & \approx\big((y\sqcup x)\sqcap(x\sqcup x)\big)\sqcap(x\sqcup y)\approx\big((y\sqcup x)\sqcap x\big)\sqcap(x\sqcup y)\approx x\sqcap(x\sqcup y)\approx \\
         & \approx x.
\end{align*}
\item If
\begin{align*}
    p(x,y,z) & :=\big((x*y)*z\big)\sqcap\big((z*y)*x\big), \\
    t_1(x,y) & :=x*y, \\
    t_2(x,y) & :=y*x, \\
s_1(x,y,z,u) & :=(x*u)\sqcap z, \\
s_2(x,y,z,u) & :=(y*z)\sqcap u
\end{align*}
then by Theorem~\ref{th4} we obtain
\begin{align*}
                   p(x,x,y) & \approx\big((x*x)*y\big)\sqcap\big((y*x)*x\big)\approx(1*y)\sqcap\big((y*x)*x\big)\approx \\
					                  & \approx y\sqcap\big((y*x)*x\big)\approx y, \\
                   p(y,x,x) & \approx\big((y*x)*x\big)\sqcap\big((x*x)*y\big)\approx\big((y*x)*x\big)\sqcap(1*y)\approx \\
					                  & \approx\big((y*x)*x\big)\sqcap y\approx y, \\ 
                   t_1(x,x) & \approx x*x\approx1, \\
                   t_2(x,x) & \approx x*x\approx1, \\
s_1\big(t_1(x,y),1,x,y\big) & \approx\big((x*y)*y\big)\sqcap x\approx x, \\
s_1\big(1,t_1(x,y),x,y\big) & \approx(1*y)\sqcap x\approx y\sqcap x\approx x\sqcap y\approx(1*x)\sqcap y\approx s_2\big(t_2(x,y),1,x,y\big), \\
s_2\big(1,t_2(x,y),x,y\big) & \approx\big((y*x)*x\big)\sqcap y\approx y.
\end{align*}
\item This follows like in (i).
\item This follows like in (i) and (ii) using Theorem~\ref{th5} instead of Theorem~\ref{th4}.
\item This follows like in (i) and (ii) using Theorem~\ref{th5} instead of Theorem~\ref{th4}.
\end{enumerate}
\end{proof}

\section{Conclusions}

The advantage of our approach of introducing algebras $\mathbf A$ closely related to given posets $\mathbf P$ is that we can introduce congruences in $\mathbf P$ by means of $\mathbf A$. Namely, we may introduce congruences on $\mathbf P$ as congruences on a fixed algebra $\mathbf A$ assigned to $\mathbf P$. Of course, the assignment of $\mathbf A$ to $\mathbf P$ is in general not unique as mentioned above and hence for different algebras assigned to $\mathbf P$ we could obtain different congruences on $\mathbf P$. However, regardless which algebra $\mathbf A$ is assigned to $\mathbf P$, the lattice of congruences on $\mathbf P$ will be distributive if $\mathbf A$ is a $\lambda$-lattice and congruences on $\mathbf P$ will permute and will be fully determined by their $1$-class if $\mathbf P$ is down-directed and either relatively pseudocomplemented or strongly sectionally pseudocomplemented. This may shed new light on the concept of congruences on posets.

Another application of our approach may be the following. Given a directed relatively or sectionally pseudocomplemented poset $\mathbf P$, we can ask if there are posets $\mathbf P_1$ and $\mathbf P_2$ of the same kind as $\mathbf P$ such that $\mathbf P_1\times\mathbf P_2\cong\mathbf P$. For posets, it is not easy to decide if such a decomposition is possible. However, we can assign to $\mathbf P$ an algebra $\mathbf A$ as constructed above. Now we have a simple criterion for deciding if there exist algebras $\mathbf A_1$ and $\mathbf A_2$ with $\mathbf A_1\times\mathbf A_2\cong\mathbf A$. Namely, such a decomposition is possible if and only if there exist $\Theta,\Phi\in\Con\mathbf A$ such that
\begin{align*}
 \Theta\vee\Phi & =\nabla_A, \\
 \Theta\cap\Phi & =\Delta_A, \\
\Theta\circ\Phi & =\Phi\circ\Theta.
\end{align*}
Hence, if such congruences exist, we have $\mathbf A\cong\mathbf A_1\times\mathbf A_2$ with $\mathbf A_1=\mathbf A/\Theta$ and $\mathbf A_2=\mathbf A/\Phi$. Now we can conversely assign posets $\mathbf P_i$ to $\mathbf A_i$ for $i=1,2$ and we can write $\mathbf P_1\times\mathbf P_2\cong\mathbf P$.

Moreover, if $\mathbf P=\mathbf P_1\times\mathbf P_2$ then we can ask if a given congruence $\Theta$ on $\mathbf P$ is directly decomposable, i.e.\ if there are $\Theta_i\in\Con\mathbf P_i$ for $i=1,2$ such that $\Theta_1\times\Theta_2=\Theta$. If for assigned algebras we have $\mathbf A=\mathbf A_1\times\mathbf A_2$ then, provided $\mathbf A$ is congruence distributive, there exist $\Theta_i\in\Con\mathbf A_i$ for $i=1,2$ such that $\Theta_1\times\Theta_2=\Theta$ and hence $\mathbf P$ has directly decomposable congruences.

Authors' addresses:

Ivan Chajda \\
Palack\'y University Olomouc \\
Faculty of Science \\
Department of Algebra and Geometry \\
17.\ listopadu 12 \\
771 46 Olomouc \\
Czech Republic \\
ivan.chajda@upol.cz

Helmut L\"anger \\
TU Wien \\
Faculty of Mathematics and Geoinformation \\
Institute of Discrete Mathematics and Geometry \\
Wiedner Hauptstra\ss e 8-10 \\
1040 Vienna \\
Austria, and \\
Palack\'y University Olomouc \\
Faculty of Science \\
Department of Algebra and Geometry \\
17.\ listopadu 12 \\
771 46 Olomouc \\
Czech Republic \\
helmut.laenger@tuwien.ac.at


\begin{thebibliography}{99}
\bibitem C
I.~Chajda, An extension of relative pseudocomplementation to non-distributive lattices. Acta Sci.\ Math.\ (Szeged) {\bf69} (2003), 491--496.
\bibitem{CEL}
I.~Chajda, G.~Eigenthaler and H.~L\"anger, Congruence Classes in Universal Algebra. Heldermann, Lemgo 2012. ISBN 3-88538-226-1.
\bibitem{CK}
I.~Chajda and M.~Kola\v r\'ik, Variety of orthomodular posets. Miskolc Math.\ Notes {\bf15} (2014), 361--371.
\bibitem{CKL}
I.~Chajda, M.~Kola\v r\'ik and H.~L\"anger, Varieties corresponding to classes of complemented posets. Miskolc Math. Notes (to appear).	http://arxiv.org/abs/1911.05138.
\bibitem{CL11}
I.~Chajda and H.~L\"anger, Directoids. An Algebraic Approach to Ordered Sers. Heldermann, Lemgo 2011. ISBN 978-3-88538-232-4.
\bibitem{CL18}
I.~Chajda and H.~L\"anger, Relatively pseudocomplemented posets. Math.\ Bohem.\ {\bf143} (2018), 89--97.
\bibitem{CLP20}
I.~Chajda, H.~L\"anger and J.~Paseka, Algebraic aspects of relatively pseudocomplemented posets. Order {\bf37} (2020), 1--29.
\bibitem{CLP21}
I.~Chajda, H.~L\"anger and J.~Paseka, Sectionally pseudocomplemented posets. Order (2021). http://arxiv.org/abs/1905.09343.
\bibitem F
O.~Frink, Pseudo-complementes in semi-lattices. Duke Math.\ J.\ {\bf29} (1962), 505--514.
\bibitem{JQ}
J.~Je\v zek and R.~Quackenbush, Directoids: algebraic models of up-directed sets. Algebra Universalis {\bf27} (1990), 49--69.
\bibitem K
P.~K\"ohler, Brouwerian semilattices. Trans.\ Amer.\ Math.\ Soc.\ {\bf268} (1981), 103--126.
\bibitem N
J.~Niederle, On pseudocomplemented and Stone ordered sets. Order {\bf18} (2001), 161--170.
\bibitem R
S.~Rudeanu, On relatively pseudocomplemented posets and Hilbert algebras. An.\ Stiint.\ Univ.\ Al.\ I.\ Cuza Iasi Sect.\ I a Mat.\ {\bf31} (1985), suppl., 74--77.
\bibitem S
V.~Sn\'a\v sel, $\lambda$-lattices. Math.\ Bohem.\ {\bf 122} (1997), 267--272.
\bibitem V
P.~V.~Venkatanarasimhan, Pseudo-complements in posets. Proc.\ Amer.\ Math.\ Soc.\ {\bf28} (1971), 9--17.
\end{thebibliography}
\end{document}